\documentclass[12pt]{article}
\usepackage{amsmath,amsfonts,enumerate,amssymb,amsthm,color}
\usepackage{mathrsfs}
\usepackage{epic,eepic,epsf,epsfig}
\usepackage{multirow}
\usepackage{amsmath}
\usepackage{amssymb}
\usepackage{amsbsy}
\usepackage{graphicx}
\usepackage{amsfonts}

\newtheorem{lem}{Lemma}[section]
\newtheorem{thm}[lem]{Theorem}

\newtheorem{cor}[lem]{Corollary}

\newtheorem{prob}[lem]{Problem}
\newtheorem{prop}[lem]{Proposition}

\def\a{\alpha}\def\b{\beta}\def\d{\delta}\def\r{\rho}\def\s{\sigma}
\def\BC{{\cal BC}}
\def\C{\mathcal{C}}
\def\G{\Gamma}
\def\R{\mathfrak{R}}
\def\S{\mathfrak{S}}
\def\Z{\mathbb{Z}}
\def\un{\underline}
\def\Aut{\hbox{\rm Aut}}
\def\BiCay{\hbox{\rm BiCay}}
\def\Cay{\hbox{\rm Cay}}
\def\H{\hbox{\rm H}}
\def\GL{\hbox{\rm GL}}
\def\PSU{\hbox{\rm PSU}}

\newcommand{\sg}[1]{\langle #1 \rangle}

\begin{document}
\title{On groups all of whose Haar graphs are Cayley graphs}
\footnotetext{The first author was partially supported by the National Natural Science Foundation of China (11571035, 11231008) and by the 111 Project of China (B16002). The second author
was partially supported by the Slovenian Research Agency (research program P1-0285, and research projects N1-0032, N1-0038, J1-5433 and J1-6720). The third author was partially supported by the Project funded by China Postdoctoral Science Foundation (2016M600841).

{\em E-mail addresses:} yqfeng@bjtu.edu.cn (Y.-Q.~Feng), istvan.kovacs@upr.si (I.~Kov\'acs),
yangdawei@ math.pku.edu.cn (D.-W.~Yang). }
\author{Yan-Quan~Feng$^{\, a},$ Istv\'an~Kov\'acs$^{\, b},$
Da-Wei Yang$^{\, c}$ \\ [+1ex]
{\small\em $^a$ Department of Mathematics, Beijing Jiaotong University, Beijing, 100044, P.~R.~China} \\
{\small\em $^b$ IAM and FAMNIT, University of Primorska,
Glakolja\v{s}ka 8, 6000 Koper, Slovenia}\\
{\small\em $^c$ School of Mathematical Sciences, Peking University, Beijing, 100871, P.~R.~China}
}
\date{}
\maketitle

\begin{abstract}
A Cayley graph of a group $H$ is a finite simple graph $\G$ such
that $\Aut(\G)$ contains a subgroup isomorphic to $H$ acting regularly on $V(\G),$ while a Haar graph of $H$ is
a finite simple bipartite graph $\Sigma$ such that $\Aut(\Sigma)$ contains a subgroup isomorphic to $H$ acting semiregularly on
$V(\Sigma)$ and the $H$-orbits are equal to the bipartite sets of
$\Sigma$. A Cayley graph is a Haar graph exactly when it is bipartite, but no simple condition is known for a Haar graph to be a Cayley graph.
In this paper, we show that the groups $D_6, \, D_8, \, D_{10}$ and $Q_8$ are the only finite inner abelian groups all of whose Haar graphs are Cayley graphs (a group is called inner abelian if it is non-abelian, but all of its proper subgroups are abelian). As an application, it is also shown that every non-solvable group has a Haar graph which is not a Cayley graph.
\end{abstract}

\section{Introduction}

Let $H$ be a finite group, and let $R$, $L$ and $S$ be three subsets of $H$ such that $R^{-1}=R$, $L^{-1}=L$, and $R\cup L$
does not contain the identity element $1$ of $H$.
The {\em Cayley graph} of $H$ relative to the subset $R,$ denoted by $\Cay(H,R),$ is the graph having vertex set
$H,$ and edge set $\{ \{h,xh\} : x \in R, h \in H \};$ and the {\em bi-Cayley graph} of $H$ relative to the triple $(R, L, S)$, denoted by $\BiCay(H, R, L, S)$, is the graph having vertex set the union of the right part
$H_0=\{h_0 : h\in H\}$ and the left part $H_1=\{h_1 : h\in H\}$, and edge set being the union of the following three sets
\begin{itemize}
\item $\big\{ \{h_0,(xh)_0\} : x \in R, \, h \in H \big\}$ (right edges),
\item $\big\{ \{h_1,(xh)_1\}:  x \in L, \, h \in H \big\}$ (left edges),
\item $\big\{ \{h_0,(xh)_1\} : x \in S, \, h \in H\big\}$ (spokes).
\end{itemize}
In the particular case when $R=L=\emptyset,$ the bi-Cayley graph
$\BiCay(H, \emptyset, \emptyset, S)$ is also known as a {\em Haar graph,} denoted by $\H(H,S)$. This name is due to Hladnik et al.~\cite{HMP}, who studied Haar graphs of cyclic groups.

Symmetries of Cayley graphs have always been an active topic among algebraic combinatorists, and lately, the symmetries of bi-Cayley graphs have also received considerable attention. For various results and constructions in connection with bi-Cayley graphs and their automorphisms, we refer the reader to \cite{AHK,AKKMM,EP,LWX1,ZF,ZF1} and all the references therein. In particular, Est\'elyi and Pisanski~\cite{EP} initiated the investigation the relationship between Cayley graphs and Haar graphs. A Cayley graph is a Haar graph exactly when it is bipartite, but no simple condition is known for a Haar graph to be a Cayley graph.
An elementary argument shows that every Haar graph of every  abelian group is also a Cayley graph (this also follows from
Proposition~\ref{ZF}).  On the other hand, Lu et al.~\cite{LWX} have constructed cubic semi-symmetric graphs, that is, edge- but not vertex-transitive graphs, as Haar graphs of alternating groups.
Clearly, as these graphs are not vertex-transitive, they are examples of Haar graphs which are not Cayley graphs.
Motivated by these observations, the following problem was posed.

\begin{prob}\label{EP1}
{\rm (\cite[Problem~1]{EP})} Determine the finite non-abelian groups $H$ for which all Haar graphs  $\H(H,S)$ are Cayley graphs.
\end{prob}

Est\'elyi and Pisanski also solved the above problem for dihedral
groups, and they showed that all Haar graphs of the dihedral group
$D_{2n}$ are Cayley graphs if and only if $n=2,3,4,5$ (see \cite[Theorem~8]{EP}). We denote by
$D_{2n}$ the dihedral group of order $2n,$ and by $Q_8$ the quaternion group.

Our goal in this paper is to solve Problem~\ref{EP1} for the class of {\em inner abelian} groups. Recall that, a group $H$ is called inner abelian if $H$ is non-abelian, and all proper subgroups of $H$ are abelian. Our main result is the following theorem.

\begin{thm}\label{1}
Let $H$ be a finite inner abelian group for which all Haar graphs $\H(H,S)$ are Cayley graphs. Then $H$ is isomorphic to
$D_6, \, D_8, \, D_{10}$ or $Q_8$.
\end{thm}

The rest of the paper is organized as follows.  In the next section we collect all concepts and results that will be used later. In Section~3, we give some preparatory lemmas, then prove Theorem~\ref{1} in Section~4.  We finish the paper with an application, namely, it will be shown that every non-solvable group has a Haar graph which is not a Cayley graph, and by this Problem~\ref{EP1} will be
reduced to the class of solvable groups.

\section{Preliminaries}

All groups in this paper are finite and all graphs
are finite and undirected. For a graph $\G$ we denote by
$V(\G),$ $E(\G)$ and $\Aut(\G)$ the vertex set, the edge set and the group of all automorphisms of $\G$. Given a vertex $v \in V(\G),$ we
denote by $\G(v)$ the set of vertices adjacent to $v,$ and if
$G \le \Aut(\G),$ then by $G_v$ the stabilizer of $v$ in $G$.

Let $\G=\H(H,S)$ be a Haar graph of a group $H$ with
identity element $1$. By \cite[Lemma~3.1(2)]{ZF}, up to graph isomorphism, we may always assume that $1\in S$. The graph
$\G$ is then connected exactly when $H=\sg{S}$.
For $g \in H$, the {\em right translation} $R(g)$ is the permutation
of $H$ defined by $R(g) : h \mapsto hg$ for $h \in H,$ and the
{\em left translation} $L(g)$ is
the permutation of $H$ defined by
$L(g) : h \mapsto g^{-1}h$ for $h \in H$.
Set $R(H):=\{ R(h) \, : \, h \in H\}$. It is easy to see that $R(H)$ can also be regarded as a group of automorphisms of $\H(H,S)$ acting on its vertices by the rule
$$
\forall i\in \{0,1\},~\forall h,g\in H:~
h_i^{R(g)}=(hg)_i.
$$

For an automorphism $\a \in \Aut(H)$ and $x, y, g \in H$,
define two permutations on $V(\G)=H_0 \cup H_1$ as follows
\begin{eqnarray}
\forall h\in H &:& h_0^{\d_{\a,x,y}}=(xh^\a)_1,~h_1^{\d_{\a,x,y}}=(yh^\a)_0, \label{Eq-delta} \\
\forall h\in H &:&  h_0^{\s_{\a,g}}=(h^\a)_0,~h_1^{\s_{\a,g}}=(gh^\a)_1. \label{Eq-sigma}
\end{eqnarray}
Set
\begin{eqnarray*}
{\rm I} &=& \{\d_{\a,x,y} : \a \in\Aut(H),~S^\a=y^{-1}S^{-1}x\},
\label{Eq-d}
\\
{\rm F} &=& \{\s_{\a,g} : \a \in \Aut(H),~S^\a=g^{-1}S \}.
\label{Eq-s}
\end{eqnarray*}
By~\cite[Lemma~3.3]{ZF}, ${\rm F}\leq \Aut(\G)_{1_0}$.
If $\G$ is connected, then ${\rm F}$ acts on the set $\G(1_0)$ consisting of all neighbours of $1_0$ faithfully. By~\cite[Theorem~1.1 and Lemma~3.2]{ZF}, we have the following proposition.

\begin{prop}\label{ZF}
{\rm (\cite{ZF})}
Let $\G=\H(H,S)$ be a connected Haar graph, and let $A=\Aut(\G)$.
\begin{enumerate}[(i)]
\item If ${\rm I}=\emptyset,$ then the normalizer
$N_{A}(R(H))=R(H)\rtimes {\rm F}$.
\item If ${\rm I}\neq \emptyset$,
then $N_{A}(R(H))=R(H)\sg{{\rm F},\d_{\a,x,y}}$
for $\d_{\alpha,x,y}\in {\rm I}$.
\end{enumerate}
Moreover, $\sg{R(H),\d_{\a,x,y}}$ acts transitively on $V(\G)$ for any
$\d_{\a,x,y}\in {\rm I}$.
\end{prop}

Let $H$ be a permutation group on a finite set $\Omega$.
For convenience, let $\Omega=\{1,\ldots,n\}$. Let $G$ be a permutation group on a finite set $\Delta,$ and let $N=G \times \cdots \times G$ with $n$ factors. We define the action of $H$ on $N$ by letting
$$
\forall g_i \in G,~\forall h\in H:~(g_1,\ldots,g_n)^h=(g_{1^{h^{-1}}},\ldots,g_{n^{h^{-1}}}).
$$

The semidirect product of $N$ by the group $H$ with respect to the above action is called the {\em wreath product} of $G$ and $H,$  denoted by  $G \wr H$. The group $G \wr H$ can be viewed as a
permutation group of the set $\Omega \times \Delta,$ by letting
the element $(g_1,\ldots,g_n;h) \in G \wr H$ act as
$$
\forall~(i,\d) \in \Omega \times \Delta:~(i,\d)^{(g_1,\ldots,g_n;h)}=(i^h,\d^{g_i}).
$$

Notice that, if $H$ is intransitive on $\Omega,$ then the group
$G \wr H$ is obviously intransitive on $\Omega \times \Delta$. This observation will be used in the next section.

Given two graphs $\G_1$ and $\G_2,$ the {\em lexicographical product}
$\G_1[\G_2]$ is defined to be the graph with vertex set
$V(\G_1) \times V(\G_2),$ and two vertices $(u_1,u_2)$ and
$(v_1,v_2)$ are adjacent in $\G_1[\G_2]$  if and only if $u_1=v_1$ and
$u_2$ is adjacent to $v_2$ in $\G_2,$ or $u_1$ is adjacent to $v_1$ in $\G_1$.  In view of \cite[Theorem]{S}, we have the following
proposition.

\begin{prop}\label{S}
{\rm (\cite{S})} Let $\G_1$ and $\G_2$ be two graphs. Then
$\Aut(\G_1[\G_2])=\Aut(\G_2) \wr \Aut(\G_1)$ if and only if the following conditions hold:
\begin{enumerate}[(i)]
\item If there exist two distinct vertices $u,v \in V(\G_1)$ such that
$\G_1(u)=\G_1(v),$ then $\G_2$ is connected.
\item If there exist two distinct vertices $u,v \in V(\G_1)$ such that
$\G_1(u) \cup \{u\}=\G_1(v) \cup \{v\},$ then the complement $(\G_2)^c$ of the graph $\G_2$ is connected.
\end{enumerate}
\end{prop}

Let $\Z H$ denote the group ring of $H$ over the ring
of integers. We denote by $\cdot$ the usual multiplication and by $\circ$ the {\em Schur-Hadamard} multiplication of $\Z H,$ that is,
\begin{eqnarray*}
\sum_{h\in H}c_h h \cdot \sum_{h\in H}d_h h &=&
\sum_{h\in H}\big(\sum_{g \in H}c_gd_{g^{-1}h} \big)h, \\
\sum_{h\in H}c_h h \circ \sum_{h\in H}d_h h &=&
\sum_{h\in H}(c_hd_h) h.
\end{eqnarray*}
Given a subset $S \subseteq H,$ let $\un{S}$ denote the $\Z H$-element $\sum_{h\in S}h$.  After Wielandt~\cite{W}, we call such elements \emph{simple quantities} (see \cite[p.~54]{W}).

Let $G$ be a permutation group of $H$ such that $R(H) \le G,$ and let $G_1$ denote the stabilizer of the identity element $1$ in $G$.
Schur~\cite{Sch} proved that the $\Z$-module spanned by all simple quantities $\un{X}$ where $X$ runs over the set of all $G_1$-orbits is a subring of $\Z H$ (also see~\cite[Theorem~24.1]{W}). This
$\Z$-module is called the {\em transitivity module} of $G_1,$  denoted by $\R(H,G_1)$. The simple quantity $\un{X}$ for a $G_1$-orbit $X$, will also be called a {\em basic quantity} of $\R(H,G_1)$. A couple of properties of transitivity modules are listed in the proposition below.
In fact, the statements in (i)-(iii) are from \cite[Propositions~22.1, 22.4 and 23.6]{W}.

\begin{prop}\label{P-tmodule1}
{{\rm (\cite{W})}}
Let $\S=\R(H,G_1)$ be a transitivity module of a group $G_1$.
The following properties hold:
\begin{enumerate}[(i)]
\item If $\sum_{h \in H}c_h h \in \S$ and $c \in \Z,$ then the simple
quantity $\un{\{h \in H : c_h=c \}} \in \S$.
\item $\S$ is closed under the Schur-Hadamard product of $\Z H$.
\item If $\un{S} \in \S$ for a subset $S \subseteq H,$ then
$\un{\sg{S}} \in \S$.
\end{enumerate}
\end{prop}

The next properties are widely used and also easy to show,
nevertheless, for easier reading we present a proof.

\begin{prop}\label{P-tmodule2}
Let $\S=\R(H,G_1)$ be a transitivity module, and $S \subseteq H$ be
a subset such that $\un{S} \in \S$. The following properties hold:
\begin{enumerate}[(i)]
\item $\un{S^{-1}}  \in \S$.
\item If $\un{\{h\}} \in \S$ for some $h\in H$ and $\un{S}$ is a basic quantity of $\S,$ then $\un{h S}$ and $\un{S h}$ are also basic quantities of $\S$.
\item $\sg{S}$ is block of imprimitivity for $G$.
\item If $\un{\{h\}} \in \S,$ then $L(h) \in C_{S_H}(G)$.
\end{enumerate}
\end{prop}

\begin{proof}
(i): It is well-known that there is a one-to-one correspondence between  the $G_1$-orbits on $H$ and the $G$-orbits on $H \times H$ given as
follows: If $\Delta$ is the $G$-orbit of $(h,k) \in H \times H,$ then the
corresponding $G_1$-orbit is equal to the set
$\{x \in H : (1,x) \in \Delta\}$. Denote the latter set by $T,$ so
in other words, $\un{T}$ is a basic quantity of $\S$. Let $T'$ be
the $G_1$-orbit corresponding to the $G$-orbit of $(k,h)$ on $H \times H$. Then $T'$
can be expressed as $T'=\{x \in H : (x,1) \in \Delta \}$. On the other
hand, using that $R(H) \le G,$ we can write
$$
(x,1) \in \Delta \iff (1,x^{-1})=(x,1)^{R(x^{-1})} \in \Delta
\iff x \in T^{-1}.
$$
This shows that $T'=T^{-1}$. Since $\un{S}=\sum \un{T_i}$ for
some basic quantities $\un{T_i},$  by the previous observation we get
$\un{S^{-1}}=\sum \un{T_i^{-1}} \in \S,$ and (i) follows.

(ii): In the group ring $\Z H$ it holds $\un{h S}=\un{\{h\}} \cdot
\un{S},$ and as $\S$ is a subring of $\Z H,$ $\un{h S}
\in \S$. Choose a basic quantity $\un{T}$ of $\S$ with $T\subseteq hS$.  By (i), $\un{h^{-1}} \in \S$.
If $T\ne h S,$ then $h^{-1}T \subsetneq S,$ contradicting that
$\un{S}$ was chosen to be a basic quantity, and (ii) follows for $h S$.
The proof of the other statement with $S h$ goes in the same way.

(iii): Let $K=\sg{S}$. By Proposition~\ref{P-tmodule1}~(iii),
$\un{K}\in \S,$ or in other words, $G_1$ fixes $K$ setwise. Let
$G_{\{K\}}$ denote the setwise stabilizer of $K$ in $G$. Then
$G_{\{K\}}=G_1R(K)$. In particular, $G_1 \le G_{\{K\}},$
and the $G_{\{K\}}$-orbit of $1$ is a block of imprimitivity for $G$
(see \cite[Theorem~1.5A]{DM}). As the latter orbit is $K,$ (iii) follows.

(iv): Let $x \in H$ and $g \in G$.  Observe that, we can choose
elements $y \in H$ and $g_1\in G_1$ such that $R(x) g=g_1 R(y)$.
By (i), $\un{\{h^{-1}\}} \in \S,$ and hence $(h^{-1})^{g_1}=h^{-1}$.
Thus, $(h^{-1}x)^g=(h^{-1})^{g_1 R(y)}=h^{-1}y$ and $x^g=1^{g_1 R(y)}=y$. These give $(h^{-1}x)^g=h^{-1}x^g,$ that is, $L(h)$ commutes with $g$.
\end{proof}

We would like to remark that the transitivity module $\R(H,G_1)$ is an example of the so called {\em S-rings} over the group $H$. For more information on these rings, we refer the reader to
\cite[Chapter~IV]{W}, and for a survey on various
applications of S-rings in algebraic graph theory,  we refer
to \cite{MP}.

\section{Properties of Haar Cayley graphs}

In this section, we give two lemmas about groups all of whose Haar  graphs are Cayley graphs. We introduce the following
notation which will be used throughout the paper.
$$
\BC=\big\{ H \; \text{is a finite group} : \forall S \subseteq H,~~\H(H,S) \; \text{is a Cayley graph} \big\}.
$$

\begin{lem}\label{L1}
The class $\BC$ is closed under taking subgroups.
\end{lem}

\begin{proof}
Let $H$ be a group in the class $\mathcal{BC}$, and let $K\leq H$.
Let $\H(K,S)$ be a Haar graph of $K$ for some subset $1\in S\subseteq K$. It is sufficient to prove that $\H(K,S)$ is a Cayley graph.
Note that $\H(K,S)$ is a union of some components of $\H(H,S)$
each being isomorphic to $\H(\sg{S},S)$.
Denote by $\G$ a component of $\H(K,S)$.
Then $\G \cong \H(\sg{S},S)$ and $V(\G)$
is a block of imprimitivity for $\Aut(\H(H,S))$.

First, we prove that $\G$ is a Cayley graph.
Since $H\in \cal{BC}$, the Haar graph $\H(H,S)$ is
a Cayley graph, and thus $\Aut(\H(H,S))$ has a subgroup $R$ acting regularly on $V(\H(H,S))$.
Since $V(\G)$ is a block of imprimitivity, the setwise stabilizer
$R_{\{V(\G)\}}$
is transitive on $V(\G)$. This implies that
$R_{\{V(\G)\}}$ is regular on $V(\G),$ and thus
$\G$ is a Cayley graph.

Now, we prove that $\H(K,S)$ is a Cayley graph.
Noting that each component of $\H(K,S)$ is isomorphic to
$\G$, assume that $\H(K,S)$ has $m$ components, and
one may identify $\{(i,u)~|~1\leq i\leq m, u\in V(\G)\}$ and
$\{\{(i,u),(i,v)\}~|~1\leq i\leq m, \{u,v\}\in E(\G)\}$ with the vertex set and the edge set of $\H(K,S)$, respectively. Since $\G$ is a Cayley graph,
$\Aut(\G)$ has a regular subgroup on $V(\G)$, say $G$.
Write $\sigma=(1\ 2\ \cdots\ m)$, the cyclic permutation on $\{1,2,\ldots,m\}$.
Then it is easy to see that the group $\sg{\sigma}\times G\cong\Z_m\times G$
acts regularly on $V(\H(K,S))$ with the action given by
$$(i,u)^{(\a,\b)}\mapsto (i^{\a},u^{\b}),~\forall (i,u)\in V(\H(K,S)),$$
where $(\a,\b)\in \sg{\sigma}\times G$. Hence $\H(K,S)$
is a Cayley graph, as required.
\end{proof}

\begin{lem}\label{L2}
Let $H$ be a group, $N \trianglelefteq H$ be a normal subgroup, and
$\bar{S} \subseteq H/N$ be a subset  such that
\begin{enumerate}[(i)]
\item $\H(H/N,\bar{S})$ is not vertex-transitive;  and
\item $\bar{S} \ne \bar{S} x$ and $\bar{S} \ne x \bar{S}$ for all non-identity element $x \in H/N$.
\end{enumerate}
Then $H$ does not belong to $\BC$.
\end{lem}

\begin{proof} Let $S=\bigcup_{aN\in \bar{S}} aN$, a subset of $H$.
Let $\G=\H(H,S)$ and $\bar{\G}=\H(H/N,\bar{S})$.
It is straightforward to show that $\G \cong \bar{\G}[nK_1]$ where $n=|N|$. Since the complement $(nK_1)^c$ of $nK_1$ is $K_n$ that is connected, by Proposition~\ref{S}, the equality
$\Aut(\G) =  \Aut(nK_1) \wr \Aut(\bar{\G})$ holds if
\begin{equation}\label{Eq-uv}
\forall u,v \in V(\bar{\G}):~
\bar{\G}(u)=\bar{\G}(v) \Rightarrow u = v.
\end{equation}

Suppose that $\bar{\G}(u)=\bar{\G}(v)$ for the vertices $u=x_i$ and $v=y_j,$ where $x, y \in H/N$ and $i,j \in \{0,1\}$.
It is easy to see that this implies that $i=j=0$ and $\bar{S} x=\bar{S} y,$ or $i=j=1$ and $\bar{S}^{-1} x = \bar{S}^{-1} y$. Thus,  $\bar{S}=\bar{S} x y^{-1}$ or $\bar{S}=x y^{-1} \bar{S}$. By (ii), $x=y$ and so $u=v$, which means that Eq.~\eqref{Eq-uv} holds.

By Proposition~\ref{S}, $\Aut(\G) =  \Aut(nK_1) \wr \Aut(\bar{\G})$.  Since $\Aut(\bar{\G})$ is not transitive on $V(\bar{\G})$,  we have that $\Aut(\G)$ is not transitive on
$V(\G)$. In particular, $\G$ is not a Cayley graph, and hence $H$ is
not in $\BC$.
\end{proof}

We finish the section with a corollary of Lemmas~\ref{L1}  and
\ref{L2} which will be useful when we deal with Haar graphs of
non-solvable groups.

\begin{cor}\label{C-of-L1L2}
Let $H$ be a group with normal series
$$
1=H_0 \unlhd  H_1 \unlhd  \cdots \unlhd  H_{n-1} \unlhd H_n=H.
$$
If for some $i \in \{0,\ldots,n-1\},$ there exists a subset $R \subset H_{i+1}/H_i$ such that
\begin{enumerate}[(i)]
\item $\H(H_{i+1}/H_i,R)$ is not vertex-transitive; and
\item $R \ne R x$ and $R \ne x R$ for all non-identity element $x \in H_{i+1}/H_i$.
\end{enumerate}
Then $H$ does not belong to $\BC$.
\end{cor}

\section{Proof of Theorem~\ref{1}}

Throughout this section $H$ denotes an inner abelian group.
In our first lemma we consider the case when $H$ is a $p$-group.
R\'edei~\cite{R} proved that $H$ is then isomorphic to one of the following groups:
\begin{itemize}
\item The quaternion group $Q_8$,
\item $M_p(m,n)=\sg{ a,b,c \, \mid \, a^{p^m}=b^{p^n}=c^p=1,[a,b]=c=a^{p^{m-1}}}$ with $m \geq 2$ and $n \geq 1$,
\item $M_p(m,n,1)=\sg{ a,b,c \, \mid \, a^{p^m}=b^{p^n}=c^p=1, [a,b]=c, [c,a]=[c,b]=1}$ with $m \geq n$ and if $p=2$ then $m+n \geq 3$.
\end{itemize}
For the groups $H=M_p(m,n)$ or $M_p(m,n,1),$
it is easy to find that $[a,b]=c \in Z(H),$ the center of $H$. Hence $[a^i,b^j]=[a,b]^{ij}=c^{ij}$ (see~\cite[Lemma~2.2.2]{G}), and $a^ib^j=b^ja^ic^{ij}$, where $i \in \Z_{p^m}$ and $j \in \Z_{p^n}$.
\medskip

\begin{lem}\label{L-p}
Suppose that $H=M_p(m,n)$ or $M_p(m,n,1)$. Let $p\geq 3$ or $p=2$ with $m \ge 3$ and $n=1$.
Then, the Haar graph $\G=\H(H,S)$ is not vertex-transitive for
$S=\{1,a,a^{-1},b,ab\}$.
\end{lem}

\begin{proof}
If $(p,m)=(3,1)$, then $H\cong M_3(1,1,1)$,
and we check by {\sc Magma}~\cite{BCP} that the Haar graph $\H(H,S)$
is not vertex-transitive. Hence we assume that $(p,m)\neq (3,1)$.

Let $A=\Aut(\G)$ and let $\bar{A}$ be the maximal subgroup of $A$ that fixes the two bipartite sets of $\G$ setwise. We settle the lemma
in two steps.
\medskip

\noindent{\bf Claim~1.} The normalizer $N_A(R(H)) \le \bar{A}$.
\medskip

Suppose to the contrary that $N_A(R(H)) \nleq \bar{A}$.
By Proposition~\ref{ZF}, there exists a permutation $\d_{\a,x,y} \in N_A(R(H))$ for some $\a \in \Aut(H)$ and $x,y \in H$
such that $S^{\a}=y^{-1}S^{-1}x$. Since
$R(H)$ acts transitively on $H_1,$ we may further assume that $1_0^{\d_{\a,x,y}}=1_1$. By Eq.~\eqref{Eq-delta}, $1_0^{\d_{\a,x,y}}=(x1^{\a})_1=1_1,$ it follows that $x=1,$ and thus $S^{\a}=y^{-1}S^{-1}$. We can write
\begin{equation}\label{Eq-p-gp-1}
\{1^{\a},a^{\a},(a^{-1})^{\a},b^{\a},(ab)^{\a}\}=
y^{-1}\{1,a^{-1},a,b^{-1}, b^{-1}a^{-1}\},
\end{equation}
implying that $y=1$, $a^{-1}$, $a$, $b^{-1}$ or $b^{-1}a^{-1}$.
Recall that $a^ib^j=b^ja^ic^{ij}$ for $i \in\Z_{p^m}$ and $j \in \Z_{p}$.

Let $y=1$. Then $\{a^{\a},(a^{-1})^{\a},b^{\a},(ab)^{\a}\}=
\{a^{-1},a,b^{-1}, b^{-1}a^{-1}\}$ by Eq.~(\ref{Eq-p-gp-1}).
If $a^{\a}=b^{-1}$, then $(a^{-1})^{\a}=b\in \{a^{-1},a, b^{-1}a^{-1}\}$, which is impossible. Similarly,
$a^{\a}\neq b^{-1}a^{-1}$. Hence $\{a^{\a},(a^{-1})^{\a}\}=
\{a^{-1},a\}$ and $\{b^{\a},(ab)^{\a}\}=
\{b^{-1},b^{-1}a^{-1}\}$.
It follows that $a^{\a}=(ab)^{\a}\cdot (b^{-1})^{\a}=b^{-1}a^{-1}\cdot b=a^{-1}c^{-1}$ or $b^{-1}\cdot ab=ac$, and since
$a^{\a}\in \{a^{-1},a\}$ and the order of $c$ in $H$ is $p$, we have $c=a^{-2}$, forcing that $(p,m)=(2,2)$, contrary to the assumption that $m \geq 3$ if $p=2$.

Let $y=a^{-1}$ or $a$. By Eq.~(\ref{Eq-p-gp-1}), we have
$\{a^{\a},(a^{-1})^{\a},b^{\a},(ab)^{\a}\}=
\{a,a^2,ab^{-1}, b^{-1}c^{-1}\}$
or $\{a^{-1},a^{-2},a^{-1}b^{-1}, b^{-1}a^{-2}c\}$, respectively.
By a similar argument as above, we have
$\{a^{\a},(a^{-1})^{\a}\}=\{a,a^2\}$ or $\{a^{-1},a^{-2}\}$,
forcing that $p=3$ and $m=1$, contrary to the assumption that
$(p,m)\neq (3,1)$. Let $y=b^{-1}$. Then $\{a^{\a},(a^{-1})^{\a},b^{\a},(ab)^{\a}\}=
\{b,ba^{-1},ba,a^{-1}\}$, while the latter set cannot contain
both $a^{\a}$ and $(a^{-1})^{\a}$, a contradiction.
Let $y=b^{-1}a^{-1}$. Then
$\{a^{\a},(a^{-1})^{\a},b^{\a},(ab)^{\a}\}=
\{ab,bc,ba^2c,a\}$, and $\{a^{\a},(a^{-1})^{\a}\}=\{bc,ba^2c\}$.
Hence $bc\cdot ba^2c=b^2a^2c^2=1$ and $(p,m,n)=(2,1,1)$, a contradiction.
\medskip

To show that $\G=\H(H,S)$ is not vertex-transitive it is sufficient to
prove the following statement.
\medskip

\noindent{\bf Claim~2.} $A=\bar{A}$.
\medskip

Note that $R(H)$ is a $p$-subgroup of $\bar{A}$. Let $P$ be a
Sylow $p$-subgroup of $\bar{A}$ containing $R(H)$.
Assume for the moment that $P \ne R(H)$.
Then $R(H) < N_P(R(H))$ (see~\cite[Theorem~1.2.11(ii)]{G}).
By Claim~1 and Proposition~\ref{ZF}, we have
$N_P(R(H)) \le R(H)\rtimes {\rm F}$, where
${\rm F}=N_A(R(H))_{1_0}$.
Since $\sg{S}=H$, the Haar graph $\G$
is connected, and thus ${\rm F}$
acts faithfully on the neighborhood
$\G(1_0)=\{1_1,a_1,(a^{-1})_1,b_1,(ab)_1\}$.
Since $\G$ has valence $5,$ $F \le S_5,$ and we have $p=2, \, 3$ or $5$. Since $|{\rm F}|\neq 1,$ by Eq.~\eqref{Eq-sigma},
there exists a $\s_{\a,g}\in {\rm F}$
of order $p$ for some $\a \in \Aut(H)$ and $g \in H$
such that $S^{\a}=g^{-1}S$, that is,
\begin{equation}\label{eq=p-gp-2}
\{1^{\a},a^{\a},(a^{-1})^{\a},b^{\a},(ab)^{\a}\}=
g^{-1}\{1,a,a^{-1},b, ab\},
\end{equation}
implying that $g=1$, $a$, $a^{-1}$, $b$ or $ab$.

Let $g=1$. By Eq.~\eqref{eq=p-gp-2}, we have
$\{a^{\a},(a^{-1})^{\a},b^{\a},(ab)^{\a}\}=\{a,a^{-1},b, ab\}$,
and hence $\{a^{\a},(a^{-1})^{\a}\}=\{a,a^{-1}\}$
and $\{b^{\a},(ab)^{\a}\}=\{b, ab\}$.
If $b^{\a}=b$ and $(ab)^{\a}=ab$, then
$a^{\a}=(ab)^{\a}\cdot (b^{-1})^{\a}=ab\cdot b^{-1}=a$.
It is easy to check that $\s_{\a,1}$ fixes each neighbor of
$1_0$, and since $\sg{ \sigma_{\a,1}} \leq {\rm F}$
acts on the neighbors of $1_0$ faithfully, $\sigma_{\a,1}=1$,
a contradiction. If $b^{\a}=ab$ and $(ab)^{\a}=b$, then
$a^{\a}=a^{-1}$, and $\sigma_{\a,1}$
interchanges $a_1$ and $(a^{-1})_1$, forcing that
$\sigma_{\a,1}$ has order $2$ and $p=2$.
Since $b^{\a}=ab$ and $b^{2^n}=b^2=1$, we have $(ab)^2=b^2a^2c^3=1$, which is impossible.

Let $g=a$ or $a^{-1}$.
Then $\{a^{\a},(a^{-1})^{\a},b^{\a},(ab)^{\a}\}=\{a^{-1},a^{-2},a^{-1}b, b\}$
or $\{a,a^{2},ab, a^2b\}$, respectively.
It follows that $\{a^{\a},(a^{-1})^{\a}\}=\{a^{-1},a^{-2}\}$
or $\{a,a^2\}$, and hence $a^3=1$ and $(p,m)=(3,1)$, a contradiction.
Similarly, if $g=b$ or $ab$, then $\{a^{\a},(a^{-1})^{\a},b^{\a},(ab)^{\a}\}=\{b^{-1},b^{-1}a,b^{-1}a^{-1}, ac\}$ or $\{b^{-1}a^{-1},b^{-1},b^{-1}a^{-2}, a^{-1}c^{-1}\}$, and
none of them can contain both $a^{\a}$ and $(a^{-1})^{\a}$ in the same time, a contradiction. We conclude that
$R(H)=P$.

Since $\langle S\rangle=H$, $\Gamma$ is connected and so $|A:\bar{A}|=2$, implying that $\bar{A}\unlhd A$. The Frattini argument (see~\cite[Theorem~1.3.7]{G}) together with Claim~1 yields $A=\bar{A}N_A(R(H))=\bar{A},$ which completes the proof of the lemma.
\end{proof}

Now, we turn to the case when $H$ is not a $p$-group.
Miller and Moreno~\cite{MM} proved that
$$
H \cong \Z_p^n \rtimes \Z_{q^m}
$$
for distinct primes $p$ and $q$.
For the next four lemmas we set $P$ and $Q$ to be a Sylow $p$- and
$q$-subgroup of $H,$ respectively, and $b$ for a generator of $Q$. Furthermore, let $S_1 \subseteq P$ be a
subset such that $1 \in S_1$.
By \cite[Theorem~5.2.3]{G}, $P=C_P(Q) \times [P,Q]$. Since $H$ is inner  abelian, it follows that $H=[P,Q]Q,$ implying that $C_P(Q)=1$.
Also, as $P\sg{b^q} < H,$ the group $P\sg{b^q}$ is abelian.
We conclude that $b$ acts on $P$ as a fixed point free automorphism of order $q$. Therefore, for any non-identity element $a \in P,$ the
$Q$-orbit of $a$ can be written as $a^Q=\{a,a^b,\ldots,a^{b^{q-1}}\}$. Since $C_P(Q)=1,$ $\sg{a^Q,Q}$ is non-abelian, and thus
$\sg{a^Q}=P$. Observe that, since the product
$aa^b \cdots a^{b^{q-1}}$ is fixed by $b,$ it is equal to the identity
$1,$ hence the rank of $P$ is at most $q-1$.
All these yield the following conditions:
\begin{equation}\label{Eq-H}
H=\sg{a,b}, \, n < q \text{ and } q \mid (p^n-1).
\end{equation}
Finally, we set $\G=\H(H,S_1 \cup \{b\})$ with $1\in S_1\subseteq P,$  $A=\Aut(\G)$ and
$\bar{A} \le A$ for the maximal subgroup that fixes the two bipartite sets of $\G$ setwise.

\begin{lem}\label{L-faithful}
With the above notation, suppose that $q > 2$. If either  $S_1=S_1^{-1}$, or
$|S_1|=4,$ $|S_1 \cap S_1^{-1}|=3$ and $\sg{S_1}=p^2,$ then
$\bar{A}$ acts faithfully on $H_0$.
\end{lem}

\begin{proof} Let $S=S_1\cup \{b\}$.
We consider the $4$-cycles of $\G$ going through the vertex $1_0,$ and denote by $\C$ the set of all such
$4$-cycles. We claim that no $4$-cycle in $\C$ goes through the vertex $b_1$. Indeed, if such a $4$-cycle existed, then it would be in the form
$$
\big(1_0, b_1, (x^{-1}b)_0=(z^{-1}y)_0, y_1, 1_0\big)
$$
for suitable $x, y \in S_1$ and $z \in S$. Thus, $x^{-1}b=z^{-1}y,$
and since $x,y \in S_1 \subseteq P,$ it follows that $z \notin P,$ hence $z=b$. This, however, implies
$x^{-1}b^2=y^{b} \in P,$ and so $b^2=1,$ which contradicts that $b$ has order $q^m > 2$.

Suppose that $S_1=S_1^{-1}$. If $|S_1|=1$, then $S_1=\{1\}$ and the above paragraph implies that $\Gamma$ is a union of cycles of length at least $5$, which implies that $\bar{A}$ is faithful on $H_0$.  Now assume that $|S_1|\geq 2$. For any $a \in S_1$ with $a \not=1$, we find that
$(1_0,a_1,a_0,1_1) \in \C$. We conclude that $b_1$ is the only neighbor of $1_0$ which is not contained in a $4$-cycle from $\C$.
Therefore, $A_{1_0} \le A_{b_1},$ and hence $A_{h_0}=R(h)^{-1}A_{1_0}R(h) \le R(h)^{-1}A_{b_1}R(h)=
A_{(bh)_1}$ for all $h \in H$. This implies that $\bar{A}$ acts
faithfully on $H_0$.

Now, suppose that $|S_1|=4,$ $|S_1\cap S_1^{-1}|=3$ and
$|\sg{S_1}|=p^2$.
Thus $p>2,$ and $S_1=\{1,a,a^{-1},c\}$ for some $a,c \in H$ such
that $|\sg{a,c}|=p^2$. It follows that
$\C$ contains exactly two $4$-cycles $(1_0,a_1,a_0,1_1)$ and
$(1_0,a^{-1}_1,a^{-1}_0,1_1)$ for $p>3$
and exactly four $4$-cycles $(1_0,a_1,a_0,1_1)$,
$(1_0,a^{-1}_1,a^{-1}_0,1_1)$, $(1_0,a_1,a^{-1}_0,1_1)$ and
$(1_0,a^{-1}_1,a_0,1_1)$ for $p=3$. This implies in turn that
$A_{1_0} \le A_{1_1},$ $A_{h_0} \le
 A_{h_1}$ for all $h \in H,$ and
$\bar{A}$ acts faithfully on $H_0$.
\end{proof}

\begin{lem}\label{L-n=1q>2}
With the above notation, let $n=1, \, q > 2$ and $S_1=\{1,a,a^{-1}\}$
for a non-identity element $a \in P$. Then $A=\bar{A}$.
\end{lem}

\begin{proof}
Let $S=S_1 \cup \{b\},$ and hence $\G=\H(H,S)$.
Notice that, since $q>2$, it follows that $p>5$. In this case,
$a^b=a^r$ for some integer $r$ coprime to $p$ such that $r$ has multiplicative order $q$ modulo $p$.

We are going to show below that $R(P) \unlhd A$.
It is sufficient to show that $R(P) \; \mathrm{char} \; \bar{A}$.
Let $L = f^{-1} \bar{A}^{H_0} f,$ where $\bar{A}^{H_0}$ denotes the permutation group of $H_0$ induced by $\bar{A}$ acting on
$H_0,$ and $f : H_0 \mapsto H$ is the bijective
mapping  $f : h_0 \mapsto h, \, h \in H$.
By Lemma~\ref{L-faithful}, $L \cong \bar{A}$.
We warn the reader that, in what follows $R(P)$ will also denote the
permutation group of $H$ consisting of the right translations
$x \mapsto xh, \, x \in H, \ h \in P$. It is sufficient to show that
$R(P) \; \mathrm{char} \; L$.

Let us consider the transitivity module $\S=\R(H,L_1)$.
It is not hard to show, considering  the $4$-cycles of $\G$ through $1_0,$ that $\bar{A}_{1_0} \le \bar{A}_{1_1}$. Since the neighbors of $1_1$ are  $1_0, \, a_0,\, a^{-1}_0$ and $b^{-1}_0,$
we obtain that $\un{\{a,a^{-1},b^{-1}\}} \in \S$.
By Propositions~\ref{P-tmodule1}~(ii) and \ref{P-tmodule2}~(i),
$\un{\{a,a^{-1}\}} = \un{\{a,a^{-1},b^{-1}\}} \circ \un{\{a,a^{-1},b\}} \in \S,$ and thus also $\un{\{b\}} \in \S$.
By Eq.~\eqref{Eq-H}, $H=\sg{a,b}$. Thus if  $\un{\{a\}} \in \S,$ then it follows by this and
Proposition~\ref{P-tmodule2}~(ii) that $\S=\Z H$. Hence $L=R(H),$ and $R(P) \; \mathrm{char} \; L,$ as required.
Now, suppose that $\un{\{a,a^{-1}\}}$ is a basic quantity of $\S$.
By Proposition~\ref{P-tmodule2}~(iii), $P=\sg{a,a^{-1}}$ is a block of imprimitivity for $L$. Let $L_{\{P\}}$ denote the setwise stabilizer of the block $P$ in $L$. Since $L_1$ fixes $P$ setwise, it follows that $L_{\{P\}}=L_1R(P)$.
Let $\r \in L_{\{P\}}$ such that $x^\r=x$ for all $x \in P$. Since,
$\un{\{b^i\}} \in \S,$ for every $b^i \in Q,$ by Proposition~\ref{P-tmodule2}~(iv), $\r L(b^i) = L(b^i) \r,$ and we find
$(b^{-i}x)^\r=x^{L(b^i) \r}=x^{\r L(b^i)}=b^{-i}x$ for any $x\in P$. Thus $\r$ is the identity mapping,
so $L_{\{P\}}$ acts faithfully on $P,$ and $L_{\{P\}}$ can be regarded as a permutation group of $P$.
By Burnside's theorem on transitive permutation groups of degree $p$ (see \cite[Theorem~3.5B]{DM}), $L_{\{P\}}$ is
doubly transitive on $P$ or it is solvable. On the other hand,
$\{a,a^{-1}\}$ is an orbit under $L_1=(L_{\{P\}})_1$ where $|P|=p > 5$.  All these show that $L_{\{P\}}$ is a solvable group. This implies that $|(L_{\{P\}})_1|=|L_1|=2,$ and we
obtain that $R(H)$ is normal in $L$. Since $R(P)$ is characteristic in $R(H)$, it is a normal Sylow $p$-subgroup of $L$, in particular, it is characteristic in $L$. We have shown that $R(P) \unlhd A$.

Suppose to the contrary that $A \ne \bar{A}$.
Since $|\bar{A}_1|=|L_1| \le 2,$ it follows that $R(Q)$ is a Sylow $q$-subgroup of $\bar{A}$. By the Frattini argument, $N_A(R(Q)) \bar{A}=A,$ and therefore, $N_A(R(Q)) \setminus \bar{A} \ne \emptyset$. Choose $\r \in N_A(R(Q)) \setminus \bar{A}$.
Since $R(P) \unlhd A,$ $\r$ normalizes $R(P)$ as well, hence also
$R(P)R(Q)=R(H)$.
By Eq.~\eqref{Eq-delta}, $\r=\d_{\a,x,y}$ for some $\a \in \Aut(H)$ and $x,y \in H$ that satisfy  $yS^\a x^{-1}=S^{-1}$. Let $\iota_{x^{-1}}$ denote the inner automorphism of $H$
induced by $x^{-1}$. Then $y S^\a x^{-1}=yx^{-1}
S^{\a \iota_{x^{-1}}}$. Replacing $yx^{-1}$ with $y$ and
$\a \iota_{x^{-1}}$ with $\a,$ we obtain that
$y S^\a= S^{-1},$ that is,
$$
\{1^\a,a^\a,(a^{-1})^\a,b^\a\}=
y^{-1}\{1,a,a^{-1},b^{-1}\},
$$
implying that $y=1,  a,  a^{-1}$ or $b^{-1}$.
Using this condition, that $a^\a$ and $(a^{-1})^\a$ are in
$y^{-1}\{1,a,a^{-1},b^{-1}\},$ and  that $p>3,$ we obtain that
$y=1,$ and $\a$ satisfies $a^\a=a^{\pm 1}$ and $b^\a=b^{-1}$.  Recall that,  $a^b=a^r$ for some integer $r$ coprime to $p$ such that
$r$ has multiplicative order $q$ modulo $p$. We can write
 $a^{\pm r}=(a^r)^\a=(b^{-1}ab)^\a=b a^{\pm 1} b^{-1}$.
Then, $b^{-1} a^{\pm r} b=a^{\pm 1},$ and it follows that
$r^2 \equiv 1 \pmod p$. This contradicts that $r$ has multiplicative order $q > 2$ modulo $p$.  This completes the proof of the lemma.
\end{proof}

\begin{lem}\label{L-n>1p>2}
With the above notation, let $n>1, \, p > 2$ and
$S_1=\{1,a,a^{-1},a^b\}$ for a non-identity element $a \in P$. Then
$A=R(H)$.
\end{lem}

\begin{proof}
Let $S=S_1 \cup \{b\},$ and hence $\Gamma=\H(H,S)$.
We set $c=a^b$. Since $P=\sg{a^Q}$ and $n>1,$ it follows that
$c \notin \sg{a}$.

Let $L = f^{-1} \bar{A}^{H_0} f,$ where $\bar{A}^{H_0}$ denotes the permutation group of $H_0$ induced by $\bar{A}$ acting on
$H_0,$ and $f : H_0 \mapsto H$ is the bijective
mapping  $f : h_0 \mapsto h, \, h \in H$.
By Lemma~\ref{L-faithful}, $L \cong \bar{A}$.

We prove below that $L=R(H)$. Here again, $R(H)$ denotes the
the permutation group of $H$ consisting of the right translations
$x \mapsto xh, \, x,h \in H$. This is equivalent to show
that the transitivity module $\R(H,L_1)=\Z H$.
For sake of simplity we set $\S=\R(H,L_1)$.
It has been shown in the proof of Lemma~\ref{L-faithful} that
$\bar{A}_{1_0} \le\bar{A}_{1_1}$. Since the neighbors of $1_1$ are  $1_0, \, a_0,\, a^{-1}_0, \, c^{-1}_0$ and $b^{-1}_0,$
we obtain that $\un{\{a,a^{-1},c^{-1},b^{-1}\}} \in \S$.

By Propositions~\ref{P-tmodule1}~(ii), $\un{\{a,a^{-1}\}}=\un{\{a,a^{-1},c^{-1},b^{-1}\}} \circ \un{\{a,a^{-1},c,b\}} \in \S$, and hence $\un{\{b^{-1},c^{-1}\}}\in \S$ and  $\un{\{b,c\}}\in \S$ by Proposition~\ref{P-tmodule2}~(i). Since $c \notin \sg{a}$, both $b$ and $c$ are not belong to  $b^{-1}\{a,a^{-1}\}c$ and $c^{-1}\{a,a^{-1}\}b$, and since $b^{-1}\{a,a^{-1}\}b=\{c,c^{-1}\}$ and $c^{-1}\{a,a^{-1}\}c=\{a,a^{-1}\}$, we find $
(\un{\{b^{-1},c^{-1}\}} \cdot \un{\{a,a^{-1}\}} \cdot \un{\{b,c\}})
\circ \un{\{b,c\}} = \un{\{c\}}$.
Thus,  $\un{\{c\}}, \un{\{b\}} \in \S,$ and
by Proposition~\ref{P-tmodule2}~(ii),
$\un{\{c^{b^i}\}} \in \S$ for all $i \in \{0,1,\ldots,q-1\}$. As the latter elements generate $H,$ we conclude that $\S=\Z H,$ and so
$L=R(H),$ as required.

Now, since $\bar{A} \cong L,$ we obtain that
$\bar{A}=R(H),$ the permutation group of $V(\G)$ induced by the right translation.  It follows that, either $A=R(H),$ or $|A : R(H)|=2$ and
$A$ acts regularly on $V(\G)$. We finish the proof by showing that the latter possibility leads to a contradiction.

Assume that $|A : R(H)|=2$. Then, one can choose $\r \in A$ such that $1_0^\r=1_1$. Since $\r$ normalizes $R(H),$ we can write $\r=\d_{\a,x,y}$ for some $x,y \in H$ and $\a \in \Aut(H)$ that satisfy $yS^\a x^{-1}=S^{-1},$ see Proposition~\ref{ZF}. Since $1_0^\r=1_1,$ it follows that $x=1$. By Eq.~\eqref{Eq-delta}, $(1_0)^{\r^2}=y_0$ and $(1_1)^{\r^2}=(y^\a)_1$. Using these and that
$\r^2 \in R(H),$ we conclude that $\r^2=R(y)=R(y^\a)$ and so $y^\a=y$. Also, $(hy)_0=(h_0)^{\r^2}=(y h^{\a^2})_0$. It follows that
$\a^2=\iota_y,$ the inner automorphism of $H$ induced by $y$.

Since $S^\a=y^{-1}S^{-1},$ we can write
$$
\{1^\a,a^\a,(a^{-1})^\a,c^\alpha,b^\a\}=
y^{-1}\{1,a,a^{-1},c^{-1},b^{-1}\},
$$
implying that  $y=1, a, a^{-1}, c^{-1}$ or $b^{-1}$.

Let $y=1$. Then,  $a^\a=a^{\pm 1}, \, c^\a=c^{-1}$ and $b^\a=b^{-1}$.  We can write
$c^{-1}=c^\a=(b^{-1}ab)^\a=b a^{\pm 1} b^{-1},$ from which $a^{b^2}=b^{-1} c b=a^{\pm 1}$. This implies that $b^4$ fixes $a,$ hence $q=2$. Then, by Eq.~\eqref{Eq-H}, $n < q=2,$ a contradiction.

Let $y=a^{\pm 1}$. Since $\a$ fixes $y,$ it follows that
$a^\a=a$. Thus $\{a,a^{-1}\} \subset \{a^{-1},a^{-2},$ $a^{-1}c^{-1},a^{-1}b^{-1}\}$ or $\{a,a^2,ac^{-1},ab^{-1}\}$.
Since $c \notin \sg{a},$ we find $p=3,$
$c^\a=a^{\pm 1}c^{-1}$ and
$b^\a=a^{\pm 1}b^{-1}$. We can write
$a^{\pm 1}c^{-1}=c^\a=(b^{-1}ab)^\a=b a b^{-1},$ from which $c^b=a^{-1}c^{\pm1}$.
This gives that $H=\sg{a,a^b,\ldots,a^{b^{q-1}}}=
\sg{a,c}$. Thus $n=2,$ $p^n-1=8,$  and as $q \mid (p^n-1),$
see Eq.~\eqref{Eq-H}, it follows
that $q=2$ and $n=1,$ a contradiction.

If $y=c^{-1},$ then it follows that
$\{a^\a,(a^{-1})^\a\} \not\subset c S^{-1}$.  Finally, let $y=b^{-1}$. Then $b S^{-1} \cap P=\{1\}$.
On the other hand, $a^\a \in P^\a=P,$ and so
$b S^{-1} \cap P=S^\a \cap P \ne
\{1\},$ a contradiction. This completes the proof of the lemma.
\end{proof}

\begin{lem}\label{L-n>4p=2}
With the above notation, let $n>4, \, p=2$ and
$$
S_1=\begin{cases} \{1,a,a^b,a^{b^2}\} &
\text{if } aa^{b^2} \in a^Q \\
\{1,a,a^b,a^{b^2},a^ba^{b^3}\} & \text{if } a a^{b^2} \notin a^Q
\end{cases}
$$
for a non-identity element $a \in P$. Then $A=R(H)$.
\end{lem}

\begin{proof}
Let $S=S_1 \cup \{b\},$ and hence $\Gamma=\H(H,S)$.
We set $c=a^b, \, d=a^{b^2}, \, e=a^b a^{b^3}$ and
$E=\sg{a,c,d}$. Since $n > 2,$ $E \cong \Z_2^3$. Note that,
since $\sg{a^Q}=P$ and $|P| > 16,$ the following conditions  hold:
\begin{equation}\label{Eq-cond}
\{a^{b^{-1}},d^b,d^{b^2}\} \cap E = \emptyset.
\end{equation}

Let $L = f^{-1} \bar{A}^{H_0} f,$ where $\bar{A}^{H_0}$ denotes the permutation group of $H_0$ induced by $\bar{A}$ acting on
$H_0,$ and $f : H_0 \mapsto H$ is the bijective
mapping  $f : h_0 \mapsto h, \, h \in H$.
By Lemma~\ref{L-faithful}, $L \cong \bar{A}$.

We prove below that $L=R(H)$. This is equivalent to show
that the transitivity module $\R(H,L_1)=\Z H$.
For sake of simplicity we set $\S=\R(H,L_1)$.
As before, let $\C$ denote the set of all $4$-cycles of $\G$ through $1_0$. Define the subsets of $H$ as
$$
X=\{h \in H : h_0 \in V(C) \text{ for some }  C \in \C, h\neq 1 \}, \;
\text{and} \;
Y=\{h \in H : b_1 \sim h_0 \text{ in } \Gamma, h\neq 1 \}.
$$
It is clear that $\un{X} \in \S$. Also, as $b_1$ is the only neighbour of
$1_0$ which is not contained in some $C \in \C,$ see the proof of Lemma~\ref{L-faithful}, it follows that $\un{Y} \in \S$ also holds.

Let $aa^{b^2}\in a^Q$. Then $S=\{1,a,b,c,d\}$,
$X=\{a,c,d,ac,ad,cd\}$ and so
$\un{\{acd\}}=\un{\sg{X}}-\un{X}-\un{\{1\}} \in \S$.
Moreover, $Y=\{ab,cb,db,b\}$.
Write $(\un{Y})^2=\sum_{h\in H}c_h h$. Using that $d^b \notin E,$
see Eq.~\eqref{Eq-cond}, we find that $c_{b^2}=3$ and
$c_h < 3$ for all $h \in H \setminus \{b^2\}$. By Proposition~\ref{P-tmodule1}~(i), $\un{\{b^2\}} \in \S$. Since $H=\sg{acd,b^2}$, we have $\S=\Z H,$ as required (note that $q > 2$ because of Eq.~\eqref{Eq-H} and the assumption $n > 4$).

Let $aa^{b^2}\notin a^Q$. Then $S=\{1,a,b,c,d,e\}$, $X=\{a,c,d,e,ac,ad,ae,cd,ce,de\}$ and $Y=\{ab,cb,db,eb,b\}$.
Note that $\sg{X}=\sg{a,c,d,e}=\sg{a,a^b,a^{b^2},a^{b^3}}$. Since $\sg{a^Q}=P$
and $|P|>16$, we have $a^{b^4}\notin \sg{X}$,
and thus  $e^b=da^{b^4}\notin \sg{X}$.
Using this we compute
\begin{equation}\label{Eq-comp1}
(\un{Y^{-1}} \cdot \un{Y}) \circ \un{\sg{X}}=
5 \; \un{\{1\}} + 2 \; \un{\{c,d,e,cd,ce,cde\}}.
\end{equation}
This together with Proposition~\ref{P-tmodule1}~(i)-(ii)
yields that $\un{\{c,d,e,cd,ce,cde\}}\in \S$.
It follows that $\un{\{c,d,e,cd,ce,cde\}}\circ \un{X}=\un{\{c,d,e,cd,ce\}}\in \S$
and $\un{\{cde\}}= \un{\{c,d,e,cd,ce,cde\}}-\un{\{c,d,e,cd,ce\}}\in \S$.
Since $e=(ad)^b=cd^b$ and $d^b \notin
E,$ it follows that $e \notin E$. Using this and that $Y=\{ab,cb,db,eb,b\},$ we compute
$$
(\un{Y^{-1}} \cdot \un{\{c,d,e,cd,ce\}}) \circ \un{Y^{-1}}=
3 \; \un{\{b^{-1},b^{-1}c\}} + 2 \; \un{\{b^{-1}d,b^{-1}e\}}.
$$
This shows that $\un{\{b^{-1},b^{-1}c\}}, \un{\{b^{-1}d,b^{-1}e\}} \in \S$.
Since
$(\un{\{b^{-1},b^{-1}c\}})^2 \circ (\un{\{b^{-1}d,b^{-1}e\}})^2=\{b^{-2}a\},$
we have $\un{\{b^{-2}a\}} \in \S$. Since $e\notin E$,
we have $cde\notin E$, and since $(cde)^{b^{-2}a}=ac\in E$,
we have $(cde)^{b^{-2}a}\neq cde$ and $\sg{cde,b^{-2}a}$ is non-abelian.
Hence $\sg{cde,b^{-2}a}=H$ and we have $\S=\Z H$, as required.

Now, since $\bar{A} \cong L,$ we obtain that
$\bar{A}=R(H),$ the permutation group of $V(\G)$ induced by the right translation. It follows that, either $A=R(H),$ or $|A : R(H)|=2$ and
$A$ acts regularly on $V(\G)$. As in the previous proof, we  finish the proof by showing that the latter possibility leads to a contradiction.

Assume that $|A : R(H)|=2$. It follows in the same way as in
the proof of the previous lemma that $y S^\a=S^{-1}$
for some $\a \in \Aut(H)$ and $y \in H$ such that $y^\a=y$ and $\a^2=\iota_y,$ the inner automorphism of $H$ induced by $y$. We can write
\begin{eqnarray*}
\{1^\a ,a^\a, c^\a, d^\a,b^\a\} &=&
y^{-1}\{1, a, c, d, b^{-1}\}, \ \text{or}  \\
\{1^\a ,a^\a, c^\a, d^\a, e^\a, b^\a\} &=&
y^{-1}\{1, a, c, d, e, b^{-1}\},
\end{eqnarray*}
depending whether $S=\{1,a,c,d,b\}$ or $\{1,a,c,d,e,b\}$. These
imply that  $y=1, a, c, d, e$ or $b^{-1}$.

Since $P^\a=P,$ it follows that
 $|y^{-1} S^{-1} \cap P|=|S^\alpha \cap  P| =|S \cap P| \ge 4$.
This shows that $y \ne b^{-1}$ and $b^\a=y^{-1}b^{-1}=y b^{-1}$ for some $y \in S_1$.
Let $\b$ denote the automorphism of $P$ induced
by the action of $b$ on $P$. Then, for any $x \in P,$
 $x^{\a^{-1} \b \a}=(b^{-1} x^{\a^{-1}} b)^\a=
byxyb^{-1}=x^{\b^{-1}}$. We obtain that $\b^\a=\b^{-1},$ and therefore, $\a$ maps any $Q$-orbit on $P$ to a $Q$-orbit.
\medskip

\noindent{\bf Case~1.} $S=\{1,a,b,c,d\}$.
\medskip

In this case $ad \in a^{Q}$ and $\b,$ viewed as a permutation of $P,$ is written in the form
$$
\b=(a\; c \; d\; d^b \; \ldots \; ad \; \ldots) \; \cdots .
$$

We claim that $\a$ fixes some element in $\{a,c,d\}$. This element
is $y$ if $y \ne 1$ since $y \in \{a,c,d\}$ and $y^\a=y$.
Suppose that $y=1$. Then $\a$ is an involution and
$\{a,c,d\}^\alpha=\{a,c,d\}$. This implies that one of $a,c$ and $d$
is fixed by $\a,$ and the claim follows. In particular, we have $(a^Q)^\alpha=a^Q$.

Let $c^\alpha=c$. Then $\b^{-1}=\b^\a=(a^\a\; c\; d^\a\; \ldots ),$
implying that $a^\a=d$ and $d^\a=a$. Thus, $(ad)^\a=ad,$ and we obtain that $\a$ fixes two points of the cycle $(a\; c \; d\; d^b \; \ldots \; ad \; \ldots)$. This contradicts that $\b^\a=\b^{-1},$ and the latter cycle has length $q >2$.

Let $a^\a=a$. Observe that $c^\a \in S_1 S_1 \subset E$.
On the other hand, $\b^{-1}=\b^\a=(a\; c^\a \; d^\a \; \ldots),$
implying that $c^\a=a^{b^{-1}}$. This together with the previous observation yields $a^{b^{-1}} \in E,$ a contradiction to Eq.~\eqref{Eq-cond}.

Finally, if $d^\a=d,$ then a similar argument yields $d^b=d^{\b}
=d^{(\b^{\a})^{-1}}=c^\a \in E$, which contradicts again
Eq.~\eqref{Eq-cond}.
\medskip

\noindent{\bf Case~2.} $S=\{1,a,b,c,d,e\}$.
\medskip

In this case $ad\notin a^{Q}$ and $\b$ is written in the form
$$
\b=(a\; c \; d\; d^b \; \ldots )(ad \; e \; e^b \; \ldots ) \; \cdots .
$$

It can be shown, using the same argument as in Case~1 and
that $\a$ permutes the $Q$-orbits, that
$\a$ fixes some element in the set $\{a,c,d,e\}$. If this element is
$a$ or $d,$ then we  can copy the argument used above in Case~1.

Let  $c^\a=c$. This implies as above that $a^\a=d$ and $d^\a=a$.
It follows in turn that $(ad)^\a=ad,$ $(ad)^Q$ is mapped
by $\a$ to itself, and $e^\a \ne e$ because $\a$ cannot
fix two points of the cycle $(ad\; e \; e^b \; \ldots )$.
It also follows that $y \ne e,$ and thus $y=1$ or $c$.
If $y=1,$ then $\{a,c,d\}^\a=\{a,c,d\},$ and so $e^\a=e,$ a contradiction. It follows that $y=c$. Then,
$\{a,c,d\}=\{a,c,d\}^\a \subset \{c,ca,cd,ce\}$. Since $e \notin E,$
$\{a,d\}=\{ca,cd\},$ which contradicts that $E=\sg{a,c,d} \cong \Z_2^3$.

Finally, let $e^\a=e$ and assume that none of
$a, c$ and $d$ is fixed by $\a$. It follows that $y=1$ or $e,$ and
thus $(ad)^\alpha \in E$. On the other hand, since
$\beta^\alpha= \beta^{-1}$ and $e^\alpha=e,$ it follows that $(ad)^\alpha=e^{(\b^{\a})^{-1}}=e^{\b}=e^b=(ad)^{b^2}=d d^{b^2},$  and hence $d^{b^2} \in E,$ a contradiction to Eq.~\eqref{Eq-cond}. This completes the proof of the lemma.
\end{proof}

Everything is prepared to prove Theorem~\ref{1}.
\medskip

\begin{proof}[Proof of Theorem~\ref{1}]
We show first that each of $D_6, D_8, D_{10}$ and $Q_8$ belongs to
$\BC$. For the first three groups this follows from \cite[Theorem~8]{EP}. Let $\G=\H(Q_8,S)$ be a Haar graph with $1 \in S$. If $\G$ is disconnected, then $\sg{S} < Q_8$ (see~\cite[Lemma~1(i)]{EP}), and thus $\sg{S}$ is abelian. This implies that $\H(\sg{S},S)$ is a Cayley graph. Since $\G$ is a union of some components each isomorphic to
$\H(\sg{S},S)$, the Haar graph $\G$ is also Cayley.
Hence we assume that $\G$ is connected. By {\sc Magma}~\cite{BCP},  all connected Haar graphs of $Q_8$ are Cayley graphs (note that the Haar graphs of $Q_8$ of valency $7$ or $8$ are isomorphic to $K_{8,8}-8K_2$
or $K_{8,8}$, respectively), and hence $Q_8$ is also in the class $\BC$.

Let $H$ be an inner abelian group such that $H \not\cong D_6, D_8, D_{10}$ and $Q_8$. We finish the proof by showing that $H$ does not belong to $\BC$.
\medskip

\noindent{\bf Case~1.} $H$ is a $p$-group.
\medskip

Then $H=M_p(m,n)$ or $M_p(m,n,1)$.
Now, if $p \geq 3$, then $H \notin \BC$ by Lemma~\ref{L-p}.
Assume that $p=2$. Then $m \geq 2$.

Let $m \geq 3$. We consider the subgroup $N=\sg{b^2}$.
Then $N \unlhd H$ and $H/N=\sg{ aN,bN} \cong M_2(m,1)$ or $M_2(m,1,1)$.
Let $\bar{\G}=\H(H/N,\bar{S})$ with $\bar{S}=\{N,aN,a^{-1}N,bN,abN\})$. By Lemma~\ref{L-p},
$\bar{\G}$ is not vertex-transitive.
If $\bar{S} \bar{x}=\bar{S}$ for some $\bar{x} \in H/N$,
then $\bar{S}$ is a union of some left cosets of $\sg{ \bar{x}}$
in $H/N$, and since $|\bar{S}|=5$,
we have $\bar{x}=N$ (the identity of $H/N$) or $\bar{S}=
\sg{ \bar{x}}$
is a subgroup of $H/N$. Clearly, $\bar{S}=\{N,aN,a^{-1}N,bN,abN\})$
is not a subgroup, and we get $\bar{x}=N$. By a similar argument we can obtain that $\bar{x}=N$ if $\bar{x}\bar{S}=\bar{S}$,
and then Lemma~\ref{L2} implies that $H \notin \BC$.

Let $m=2$. Recall that $H \ne M_2(2,1)$ as $M_2(2,1) \cong D_8$.
Thus, if $n=1$, then $H=M_2(2,1,1)$. In this case we obtain by
the help of {\sc Magma}~\cite{BCP} that the Haar graph
$\H(H,\{1,a,a^{-1},b,ab\})$ is not vertex-transitive.
Hence $H \notin \BC$. Let $n \geq 2$.
Then $H=M_2(2,n)$ or $M_2(2,2,1)$.
We consider the subgroup $N=\sg{ b^{2^2}}$.
Then $N \unlhd H$ and $H/N=\sg{aN, bN} \cong M_2(2,2)$ or $M_2(2,2,1)$. Let $\bar{\G}=\H(H/N,\bar{S})$ with
$$
\bar{S}=\{N,aN,bN,abN,ab^2N,ab^3N\}).
$$
By {\sc Magma}~\cite{BCP}, $\bar{\G}$ is not vertex-transitive.
If $\bar{S}\bar{x}=\bar{S}$ for some $\bar{x} \in H/N$,
then $\bar{S}$ is a union of some left cosets of $\sg{ \bar{x}}$
in $H/N$. In particular, $\bar{x} \in \bar{S}$.
Since $|\bar{S}|=6$, the order of $\bar{x}$ in $H/N$ is $1, 2, 3$ or
$6$. On the other hand, since each element in $\bar{S} \setminus \{N\}$ has order $4,$ we have $\bar{x}=N$.
Similarly, we have that $\bar{x}\bar{S} \neq \bar{S}$ for all non-identity element $\bar{x} \in H/N,$ and thus $H \notin \BC$
by Lemma~\ref{L2}. This completes the proof for $p$-groups.
\medskip

\noindent{\bf Case~2.} $H$ is not a $p$-group.
\medskip

Then $H=\Z_p^n \rtimes \Z_{q^m}$ for distinct primes $p$ and $q$.
In view of Lemmas~\ref{L-n=1q>2}--\ref{L-n>4p=2}, we may assume  that $n=1$ and $q=2,$ or $2 \le n \le 4$ and $p=2$.

Let $n=1$ and $q=2$. Let $N=\sg{b^2}$.
Then, $N \unlhd H$ and $H/N=\sg{aN, bN} \cong D_{2p}$.
Consider the Haar graph
$\bar{\G}=\H(H/N,\bar{S})$ with
$$
\bar{S}=\{N, aN, a^3N, bN, abN, a^2b N, a^4b N\}.
$$
The graph $\bar{\G}$ is not vertex-transitive for $p>5$.
This can be verified by the help of {\sc Magma}~\cite{BCP} for $p=7,$
and it was proved in \cite[Proposition~7]{EP} if $p> 7$.
If $\bar{S}=\bar{S} \bar{x}$ for some $\bar{x} \in N/H,$ then
$\bar{S}$ is a union of left cosets of $\sg{\bar{x}}$.
Since $N \in \bar{S}$ and $|\bar{S}|=7,$ it follows that
$\bar{x}=N$ or $S$ is a subgroup of $N/H$ of order $7$. Clearly, the latter option is impossible, and we get $\bar{x}=N$.
It can be shown in the same way that $\bar{S}=\bar{x} \bar{S}$ forces that $\bar{x}=N,$ and thus $H \notin \BC$ by Lemma~\ref{L2} if $p > 5$.

Now, suppose that $p=3$ or $5$. Since $H \ncong D_6, D_{10},$ we
obtain that $m \ge 2$. Let $N=\sg{b^4}$. Then,
$N \unlhd H$ and $H/N=\sg{aN, bN} \cong \Z_p \rtimes \Z_4$.
Consider the Haar graph
$\bar{\G}=\H(H/N,\bar{S})$ with
$$
\bar{S}=\{N, aN, bN, abN, ab^2N, ab^3 N\}.
$$
We find by {\sc Magma}~\cite{BCP} that $\G$ is not vertex-transitive.
If $\bar{S}=\bar{S} \bar{x}$ for some $\bar{x} \in N/H,$ then
$\bar{S}$ is a union of some left cosets of $\sg{\bar{x}}$, in particular,
$\bar{x}$ has order $1, 2, 3$ or $6$.
Since each element in $\bar{S} \setminus \{N,aN,ab^2N\}$ has order $4,$ it follows that $\bar{x}=N$.
It can be shown in the same way that $\bar{S}=\bar{x} \bar{S}$ forces that $\bar{x}=N,$ and thus $H \notin \BC$ by Lemma~\ref{L2}.

Finally, let $2 \le n \le 4$ and $p=2$. By Eq.~\eqref{Eq-H}, $n < q$ and $q \mid (2^n-1)$. It follows that $q=3$ if $n=2,$ $q=7$ if $n=3,$
and $q=5$ if $n=4$. The group $\GL(n,2)$ has Sylow $q$-subgroup of order $q$. Let $N=\sg{b^q}$. Then $N \unlhd H,$ and for $n=2, 3, 4,$ the quotient group $H/N$ is isomorphic to $\Z_2^2 \rtimes \Z_3 \cong A_4, \Z_2^3 \rtimes \Z_7$ and $\Z_2^4 \rtimes \Z_5,$ respectively,
given by the presentations
\begin{itemize}
\item $\Z_2^2 \rtimes \Z_3=
\big\langle x,y,z \mid x^2=y^2=z^3=1, [x,y]=1, x^z=y, y^z=xy
\big\rangle$,
\item $\Z_2^3 \rtimes \Z_7=
\big\langle x,y,z,v \mid x^2=y^2=z^2=u^7=1, [x,y]=[x,z]=[y,z]=1,  x^u=y, y^u=z, z^u=xy \big\rangle,$
\item $\Z_2^4 \rtimes \Z_5=
\big\langle x,y,z,v,w \mid x^2=y^2=z^2=v^2=w^5=1, [x,y]=[x,z]=[x,v]=[y,z]=[y,v]=[z,v]=1, x^w=v, y^w=xy, z^w=yz, v^w=zv \big\rangle$.
\end{itemize}

Then, we find by the help of
{\sc Magma}~\cite{BCP} that the Haar graph
$\H(H/N,S)$ is not-vertex transitive where $S=\{1,x,z,xyz\}$
for $H/N=\Z_2^2 \rtimes \Z_3,$ $S=\{1,x,u,xyu,xz u\}$ for
$H/N=\Z_2^3 \rtimes \Z_7,$ and $S=\{1,x,w,xyw,xzw\}$
for $H/N=\Z_2^4 \rtimes \Z_5$.  Furthermore, a direct computation yields $S \ne Sx$ and $S \ne xS$ for any non-identity element $x
\in H/N$. Thus, by Lemma~\ref{L2},
$H \notin \BC$. This completes the proof of the theorem.
\end{proof}

We finish the paper with an application of Theorem~\ref{1}, namely,
we show that every non-solvable group has a Haar graph which is not a Cayley graph, that is, each group in the class $\BC$ must be solvable. First, we state a corollary of Theorem~\ref{1}.

Since each non-abelian group has an inner abelian subgroup, we have the following corollary by
Theorem~\ref{1} and Lemma~\ref{L1}.

\begin{cor}\label{C-of-1}
Let $G$ be a group in the class $\BC$. Then the following hold:
\begin{enumerate}[(i)]
\item Each Sylow $p$-subgroup of $G$ with a prime $p\geq3$ is abelian.
\item If $G$ is non-abelian, then $G$ has a subgroup isomorphic to $D_6$,
$D_8$, $D_{10}$ or $Q_8$.
\end{enumerate}
\end{cor}

\begin{thm}\label{2}
Every finite non-solvable group has a non-Cayley Haar graph.
\end{thm}

\begin{proof}
Suppose to the contrary that $G$ is a non-solvable group in the class
$\BC$. We claim that $G$ contains a non-solvable $\{2,3,5\}$-subgroup $L$.

The statement holds obviously when
$G$ is a $\{2,3,5\}$-group. Now assume that $|G|$ has a prime divisor $p > 5$. Let $p_1,\ldots, p_m$ be all prime divisors of $|G|$ with $p_i>5$, and let $P_i$ be a Sylow $p_i$-subgroup of $G$ for each $1\leq i\leq m$.
By Corollary~\ref{C-of-1}~(i), $P_i$ is abelian.
Consider the group $M=P_i\sg{g},$ where $g \in N_G(P_i)$ has order $r^m$ for some prime $r$.

If $r=p_i,$ then $g \in P_i,$ and hence $M=P_i$ is abelian. If $r \ne p_i$, then $M$ has a cyclic Sylow $r$-subgroup and Corollary~\ref{C-of-1}~(ii) implies that the $\{p_i,r\}$-group $M$ ($p_i>5$) is abelian. These yield that $P_i \le Z(N_G(P_i)),$ and thus
by Burnside's $p$-complement theorem (see \cite[Theorem~7.4.3]{G}),   $G=K_iP_i,$ where $K_i \unlhd G$ and $K_i \cap P_i=1$.
Set $L=K_1\cap \cdots\cap K_m$. Then $L\unlhd G$,
and since $|L|$ divides $|K_i|$ and $p_i$ does not divide
$|K_i|$ for each $1\leq i\leq m$,
$L$ is a $\{2,3,5\}$-subgroup.
Since $G/L\lesssim G/K_1\times \cdots \times G/K_m\cong
P_1\times \cdots\times P_m$ (see~\cite[Exercises~10, page 14]{G}), $G/L$ is solvable,
and since $G$ is non-solvable, $L$ is non-solvable, as claimed
(recall that, given two groups $A$ and $B,$ we write
$A \lesssim B$ if $A$ is isomorphic to a subgroup of $B$).

Since $L$ is non-solvable, it has a composition factor $T$ which is a
non-abelian simple $\{2,3,5\}$-group. By Corollary~\ref{C-of-L1L2},
in order to arrive at a contradiction, it is enough to show that there exists a subset $S \subset T$ such that the Haar graph $\H(T,S)$ is not vertex-transitive and $S \ne Sx$ and $S \ne xS$ for any non-identity element $x \in T$.

By \cite[Theorem~I]{HL}, $T$ is isomorphic to
one of the following groups: $A_5, \, A_6$ and $\PSU(4,2)$.
It follows that $A_4 < T$. This is obvious if $T=A_5$ or $A_6,$
and can be checked for $T=\PSU(4,2)$ in \cite[p.~24]{Atlas}.

Now, consider the graph $\G=\H(T,S)$ where $S \subset A_4 < T,$
$A_4=\big\langle x,y,z \mid x^2=y^2=z^3=1, [x,y]=1, x^z=y, y^z=xy
\big\rangle$ and $S=\{1,x,z,xyz\}$. Notice that, the graph $\H(\sg{S},S)$ appeared already in the proof of Theorem~\ref{1}. We observed
that $\H(\sg{S},S)$ is not vertex-transitive, this can be checked by {\sc Magma}~\cite{BCP}, and also, $S \ne Sx$ and $S \ne xS$ for any
$x \in A_4, x \ne 1$. This implies at once that $\G$ is not vertex-transitive as well. Suppose that $S=Sx$ or $S=xS$ for $x \in T$.
Since $1 \in S,$ it follows that $x \in S \subset A_4,$ and hence $x=1$.
This completes the proof of the theorem.
\end{proof}

\end{document}